\documentclass [11pt,twoside,a4paper]{article}
\usepackage{amsfonts}
\usepackage{amsthm}
\usepackage{amsmath}
\usepackage{amssymb}
\usepackage{amscd}
\usepackage{amsbsy}            %
\usepackage{psfrag}            %
\usepackage{epsf}              %
\usepackage{graphicx}          %
\usepackage{makeidx}           %
\usepackage{color}             %
\usepackage{fancyhdr}

\newtheorem{thm}{Theorem}[section]
\newtheorem{exm}[thm]{Example}

\newtheorem{lem}[thm]{Lemma}

\newtheorem{prop}[thm]{Proposition}
\newtheorem{cor}[thm]{Corollary}

\newtheorem{conj}[thm]{Conjecture}
\newcommand{\abs}[1]{\left\vert#1\right\vert}
\newcommand{\norm}[1]{\parallel\! #1\! \parallel}
\newcommand{\seq}[1]{\left<#1\right>}
\newcommand{\set}[1]{\left\{#1\right\}}

\newcommand{\diag}{\textit{diag}}

\newcommand{\perm}{\texttt{perm}}
\newcommand{\rank}{\texttt{rank}}

\newcommand{\spann}{\texttt{span}}

\newcommand{\al}{\alpha}

\newcommand{\be}{\beta}
\newcommand{\diams}{\diamondsuit}
\newcommand{\ga}{\gamma}
\newcommand{\Ga}{\Gamma}

\newcommand{\la}{\lambda}

\newcommand{\si}{\sigma}

\newcommand{\A}{\mathcal{A}}
\newcommand{\B}{\mathcal{B}}
\newcommand{\C}{\mathcal{C}}

\newcommand{\R}{\textbf{R}}

\newcommand{\M}{\mathcal{M}}

\newcommand{\caE}{\mathcal{E}}

\newcommand{\caO}{\mathcal{O}}
\newcommand{\caS}{\mathcal{S}}
\newcommand{\cP}{\mathcal{P}}
\newcommand{\caQ}{\mathcal{Q}}

\newcommand{\K}{\mathcal{K}}

\newcommand{\caI}{\mathcal{I}}
\newcommand{\caL}{\mathcal{L}}
\newcommand{\st}{\mathcal{S}}
\newcommand{\T}{\mathcal{T}}

\newcommand{\bu}{\textbf{u}}
\newcommand{\bv}{\textbf{v}}
\newcommand{\bw}{\textbf{w}}

\newcommand{\bx}{\textbf{x}}
\newcommand{\by}{\textbf{y}}

\newcommand{\bU}{\textbf{U}}
\newcommand{\bV}{\textbf{V}}

\newcommand{\bfe}{\textbf{e}}

\def\RR{\mathbb R}

\def\CC{{\mathbb C}}

\newcommand{\mnrts}{$m$th order $n$-dimensional real tensors\ }

\newcommand{\mnsts}{$m$th order $n$-dimensional symmetric tensors\ }

\newcommand{\beq}{\begin{equation}}
\newcommand{\eeq}{\end{equation}}
\newcommand{\bey}{\begin{eqnarray}}
\newcommand{\eey}{\end{eqnarray}}
\newcommand{\beyy}{\begin{eqnarray*}}
\newcommand{\eeyy}{\end{eqnarray*}}

\title{Separable symmetric tensors and separable anti-symmetric tensors}
\author{Changqing Xu\thanks{Email: cqxurichard@mail.usts.edu.cn}\\
School of Mathematical Sciences\\
Suzhou University of Science and Technology, Suzhou, China}

  \makeatletter
      \def\@setcopyright{}
      \def\serieslogo@{}
      \makeatother
 \date{\today}
\begin{document}
\maketitle

\begin{abstract}
In this paper, we first introduce the invertibility of even-order tensors and the separable tensors, including separable symmetry tensors and separable anti-symmetry tensors,
defined respectively as the sum and the algebraic sum of rank-1 tensors generated by the tensor product of some vectors, say, $\bv_{1},\bv_{2}, \ldots, \bv_{m}$. 
We show that the $m!$ sumrands, each in form $\bv_{\si(1)}\times \bv_{\si(2)}\times\ldots\times \bv_{\si(m)}$, are linearly independent if $\bv_{1},\bv_{2}, \ldots, \bv_{m}$ are 
linearly independent, where $\si$ is any permutation on $\set{1,2,\ldots,m}$. We offer a class of tensors to achieve the upper bound for $\rank(\A) \le 6$ for all $\A\in \RR^{3\times 3\times 3}$ ( see e.g. \cite{LPZ2013}).  We also show that each $3\times 3\times 3$ anti-symmetric tensor is separable.   
\end{abstract}

\noindent \textbf{keywords:} \  S-product; invertible tensor; separable symmetric tensor; separable anti-symmetric tensor. \\
\noindent \textbf {AMS Subject Classification}: \   53A45, 15A69.  \\

%\maketitle

\section{Introduction}
\setcounter{equation}{0}

A tensor is a multi-array or hyper-matrix with multiple index, which can be applied in many fields related to multi-dimensional dimensional data (MDD) such as theory of relativity \cite{WS1972}, elasticity \cite{QZB2016} and magnetics , and computer vision \cite{SH2005, SH2006}. Recently Wang, Gu, Lee and Zhang \cite{WGLZ2021} proposed a quantum algorithm for recommendation systems by using third-order tensors. \\
\indent  A tensor can be regarded as a natural extension of a matrix which is fundamental in many fields. Some basic terminology e.g.  rank, determinant, the eigenvalues and 
the inverse have been introduced and investigated in tensor theory. For example, there are several different kind of definitions for the rank of a tensor, say, the marginal rank, meaning the maximal rank of the matricizations in different modes. There are also several kind of eigenvalues (and related eigenvectors) such as the H-eigenvalues, the Z-eigenvalues and the E-eigenvalues (see e.g. \cite{qiluo2017}).  The definition of the determinant of a hypercubic tensor, that is, a tensor with the same dimension in all modes, is much more complicate than that of a square matrix and few known work can be found relevant to this topic.\\
\indent  The rank of a tensor was proposed in 1927 by Hitchcock \cite{Hitch192701,Hitch192702} following the introduction of the polyadic form of tensors, i.e., the rank-one 
decomposition of a tensor. In 1944 Cattell \cite{Cat1944,Cat1952} presented the parallel proportional analysis and offered the idea of multiple axes for analysis, which became 
popular in 1970 when it appeared in the psychometrics community in the form of CANDECOMP (canonical decomposition) by Carroll and Chang \cite{CC1970} and PARAFAC 
by Harshman \cite{Harsh1970}, which is also referred as the CP decomposition.  A CP decomposition of a tensor $\A\in \RR^{3\times 3\times 3}$ is in form 
\beq\label{eq: cpdecomp01} 
\A=\sum\limits_{j=1}^{R} \al_{j}\times \be_{j}\times \ga_{j} 
\eeq
where $\al_{j}, \be_{j}, \ga_{j}\in \RR^{3}$ are nonzero vectors. The smallest positive integer $R$ for (\ref{eq: cpdecomp01}) is called the \emph{CP rank}, or simply the rank of 
$\A$ denoted by $\rank(\A)$.  As Kolda pointed out there are several differences between matrix rank and tensor rank,  one of which is that the rank of a real-valued tensor 
may be different over $\RR$ and $\CC$, and a significant one is that the rank-1 decomposition of a tensor is usually unique while that of a matrix is generally not. As is 
mentioned by Kolda and others, the rank of a tensor can be larger than the maximal dimensionality. For example, the rank of an $2\times 2\times 2$ tensor can be either 
2 or 3; and for an $3\times 3\times 3$ tensor $\A$, Lavrauw, Pavan and Zanella \cite{LPZ2013} shows that $\rank(\A)$ can take value in $[6]$ if $\A\neq \caO$. 
Our result on separable tensors presents a class of $3\times 3\times 3$ tensors which satisfy $\rank(\A)=6$. \\
\indent  For our convenience, we denote by $[n]$ the set $\set{1,2,\ldots, n}$ for any positive integer $n$ and $\RR$ the field of real numbers.  A tensor $\A$ of size 
$\rm{I}:=d_1\times d_2\times \ldots \times d_m$ is an \emph{$m$-array}.  $\A$ is called a \emph{hypercubic} if all its modes are of same length, i.e., 
$d_1=\ldots =d_m=[n]$.  All the tensors involved, if not otherwise mentioned, are hypercubical.  Denote by $\T_{m;n}$ the set of \mnrts and $\T_{m}$ the set of $m$-order 
tensors.  An $m$-order tensor $\A\in \T_{m;n}$ is called \emph{symmetric} if each entry $A_{i_{1}i_{2}\ldots i_{m}}$ is invariant under any permutation on its indices. Denote 
by $\st_{m;n}$ the set of \mnsts.  Let $\A\in \st_{m;n}$.  Then $\A$ is associated with an $m$-order $n$-variate homogeneous polynomial    
\beq\label{def: asspolyn}
f_{\A}(\bx):=\A\bx^m=\sum\limits_{i_1,i_2,\ldots,i_m} A_{i_1i_2\ldots i_m}x_{i_1}x_{i_2}\ldots x_{i_m}
\eeq
A symmetric tensor $\A\in \st_{m;n}$ is called a \emph{positive definite} tensor if  
\beq\label{def: psd-t}
 f_{\A}(\bx)>0,  \quad \forall 0\neq \bx\in \RR^n,
\eeq 
and is called a \emph{positive semidefinite} or psd tensor if  $f_{\A}(\bx)\ge 0$ for all $\bx\in \RR^n$. 
It is easy to see that an odd-order psd tensor is the zero tensor.  Thus we may assume in the following that $m$ is an even number, if not mentioned otherwise.  For further study on positive (semi-)definite tensors we refer the reader to \cite{CS2013, cglm2008, qieig2005,qicop2013}. \\
\indent  Let $m=2k$ ($k\in \set{1,2,3,\ldots}$).  The \emph{identity tensor} $\caI=(\delta_{\si})\in \T_{m;n}$ is defined as 
\beq\label{eq:defidentensor} 
\delta_{i_{1}i_{2}\ldots i_{k}j_{1}j_{2}\ldots j_{k}} = \delta_{i_{1}j_{1}}\delta_{i_{2}j_{2}}\ldots \delta_{i_{k}j_{k}} 
\eeq
for $\si:=(i_{1},i_{2},\ldots,i_{m})\in S(m,n)$, where $m=2k, \delta_{ij}\in \set{0,1}$, and $\delta=1$ if $i=j$. We will present some basic properties of the identity tensors in 
the next section. \\
\indent  The tensor multiplication can be defined differently with specifying the dimensions to be multiplied and the order of the dimensions of the resulted tensor. 
For example, by fixing a specific mode, say $k\in [m]$, the $k$-mode product of a tensor $\A\in \T_{m;n}$ with a matrix $M\in \RR^{n\times n}$ is defined as 
\beq\label{eq: nmodeproduct}  
(\A\times_{k} M)_{i_{1}i_{2}\ldots i_{m}} = \sum\limits_{j=1}^{n} A_{i_{1}\ldots i_{k-1}j i_{k+1}\ldots i_{m}} m_{j i_{k}}  
\eeq
the resulted tensor, denoted by $\A\times_{k} M\in \T_{m;n}$, remain the same size as that of $\A$. But this definition applies to any tensor $\A$ of any size. When 
$A\in \RR^{n\times n},B\in \RR^{n\times n}$, we have $A\times_{1} B=B^{\top} A,  A\times_{2} B=AB^{\top}$.  Note the difference between (\ref{eq: nmodeproduct}) 
and the one in \cite{K2009, K2015}.   The definition (\ref{eq: nmodeproduct}) is also meaningful when $M=\bv\in \RR^{n}$ is a vector, in which case the resulted tensor 
is of order $m-1$, and the product is called a \emph{contractive $n$-mode product} of $\A$ with vector $\bv$. \\
\indent  Another interesting tensor product is the t-product defined between two third order tensors.  Let 
\[ \A\in \RR^{n_{1}\times n_{2}\times n_{3}}, \B\in \RR^{n_{2}\times n\times n_{3}}. \]
The t-product tensor of $\A$ and $\B$, denoted $\C:=\A\ast \B$, is an $n_{1}\times n\times n_{3}$ tensor defined by
\beq\label{tprod}
C_{i_{1}i_{2}i_{3}} = \sum\limits_{j_{1},j_{2}} A_{i_{1}j_{1} j_{2}} B_{j i_{2}(i_{3}+1-j_{2})} 
\eeq
For more properties on t-product of tensors, we refer to \cite{KM2011}. \\       
\indent Let $S_{1}=\set{s_{1},s_{2},\ldots, s_{p}}, S_{2}=\set{t_{1}, t_{2},\ldots, t_{q}}$ be two subsets of $[p+q]$ with $S_{1}\cup S_{2}=[p+q]$, where $p,q$ are positive integers.  The \emph{S-product} of tensors $\A\in \T_{p;n}$ and $\B\in \T_{q;n}$, denoted $\A\boxtimes_{S} \B$, is defined by 
\beq\label{eq: S-product}
(\A\boxtimes_{S} \B)_{i_{1}i_{2}\ldots i_{m}} = \sum\limits_{k_{1},k_{2},\ldots, k_{r}}  A_{i_{1}i_{2}\ldots i_{p}}B_{j_{1}j_{2}\ldots j_{q}}
\eeq
where $m=p+q-2r,  r=\abs{S_{1}\cap S_{2}}$, and 
\[ \set{i_{1},i_{2},\ldots, i_{p}}\subseteq S_{1}, \qquad  \set{ j_{1}, j_{2},\ldots, j_{q}}\subseteq S_{2}, \]
and the sum in the right hand side of (\ref{eq: S-product}) is taken over all subscripts in $S:=\set{k_{1},k_{2},\ldots, k_{r}} = S_{1}\cap S_{2}$.  
The S-product can be contractive or extensive. In fact, there are two cases for the S-product: 
\begin{description}
\item[(a)]  $r>0$($1\le r\le \min\set{p, q}$), i.e., $S=S_{1}\cap S_{2}$ is nonempty. $\A\boxtimes_{S}\B$ is a tensor of order $p+q-2r$ by (\ref{eq: S-product}).    
\item[(b)]  $r=0$, i.e., there is no intersection between $S_{1}$ and $S_{2}$.  Then $S_{1}\cup S_{2}=[p+q]$, $\A\boxtimes_{S}\B$ is a tensor of order $p+q$. 
\end{description}
If $S_{2}\subseteq S_{1}$, then $S=S_{2}$, $\A\boxtimes_{S} \B$ is a contractive product along mode-$S_{2}$, yielding a tensor $\A\B:=\A\boxtimes_{S} \B$ of order $m=p-q$.  
A special case is the inner product $\A\B=\seq{\A, \B}$ when $p=q$ or equivalently $S_{1}=S_{2}$. On the other hand, the S-product of tensors in case (b) is exactly the outer (extensive) product of tensors. \\
\indent  The S-product can be recursively employed to yield a higher order or a lower order tensor, depending on what we need.  An extreme case is a rank-1 symmetric tensor 
$\bx^{m}$ generated from the power of a vector $\bx$ in the sense of S-product of $\bx$.  Thus the S-product unifies all possible multiplications of tensors, including the familiar contractive and the outer(tensor) product of tensors. \\
\indent In the following sections, we present some basic properties on the identity tensors. Also introduced are the separable symmetric tensors and the separable anti-symmetric tensors.  

\vskip 10pt

\section{Invertibility of a hypercubic tensor with an even-order}
\setcounter{equation}{0}

We first present some basic properties on the identity tensors.  Recall that the identity matrix $I_{n}$ in matrix space $\M_{n}$ obeys the following rule: 
\beq\label{eq: identity00}
I_{n} A = AI_{n} =A, \quad  \forall A\in \M_{n}.
\eeq
This is also valid for the identity tensor in tensor space $\T_{m;n}$. 
\begin{lem}\label{le:le01}
For any even number $m=2k$, the identity tensor $\caI$ defined by (\ref{eq:defidentensor} ) in tensor space $\T_{m; n}$ obeys the following rule: 
\beq\label{eq: identt01}
\caI \A = \A \caI =\A, \quad  \forall \A\in \T_{m;n}.
\eeq
\end{lem}

\begin{proof}
For any $\si\equiv (i_{1},i_{2},\ldots, i_{m})\in S(m,n)$, we have 
\beyy
  (\A \caI)_{i_{1}i_{2}\ldots i_{m}} &=&\sum\limits_{j_{1},j_{2},\cdots,j_{k}} A_{i_{1}i_{2}\ldots i_{k} j_{1}j_{2}\ldots j_{k}} \delta_{j_{1}j_{2}\ldots j_{k} i_{k+1}i_{k+2}\ldots i_{m}}\\
                                       &=&\sum\limits_{j_{1},j_{2},\cdots,j_{k}} A_{i_{1}i_{2}\ldots i_{k} j_{1}j_{2}\ldots j_{k}} \delta_{j_{1}i_{k+1}} \delta_{j_{2}i_{k+2}} \ldots \delta_{j_{k}i_{m}}\\
                                       &=& A_{i_{1}i_{2}\ldots i_{k} i_{k+1}i_{k+2}\ldots i_{m}} 
\eeyy
Thus $\A\caI =\A$.  Similarly we can prove $\caI \A=\A$. 
\end{proof}
\indent The identity tensor $\caI\in \T_{m;n}$ can be regarded as the S-power of the identity matrix $I_{n}$ in the sense of the outer-product: 
\[  \caI = I_{n}\boxtimes_{S_{2}} I_{n}\boxtimes_{S_{3}} \ldots \boxtimes_{S_{k}} I_{n} \]
where $S_{i}\equiv \set{i,k+i}$ ($i=1,2,\ldots,k$).  Therefore we can also write $\caI=I_{n}^{[k]}$.\\
\indent   An even-order tensor $\A\in \T_{m;n}$ is said to be \emph{invertible} if there exists a tensor $\B\in \T_{m;n}$ such that 
\beq\label{eq: invt}
\B\A = \A \B =\caI.
\eeq
$\B$ is called the inverse of $\A$ and is denoted by $\A^{-1}$. \\

\indent  The invertibility of an even-order tensor can be transferred to that of a square matrix by tensor \emph{matricization}, which is usually called \emph{unfolding}.  
An unfolding of a tensor $\A$ is defined as a process through which the elements of $\A$ is rearranged into a matrix. For more detail concerning unfolding of a general 
tensor, we refer the reader to \cite{K2009}. \\ 
\indent  Given an even-order tensor $\A\in \T_{m;n}$ with $m=2k$.  There are many ways to unfold tensor $\A$ into an $n^{k}\times n^{k}$ matrix.  A \emph{normal unfolding} 
is the process which yields a matrix $A=(a_{ij})$ whose entries are defined as $a_{ij}=A_{i_{1}i_{2}\ldots i_{k}i_{k+1}i_{k+2}\ldots i_{2k}}$ where 
\beq\label{eq: t2m}
 i = 1 + \sum\limits_{r=1}^{k} (i_{r}-1)n^{k-r},  \quad  j = 1 + \sum\limits_{r=1}^{k} (i_{k+r}-1)n^{k-r}.   
\eeq
We call matrix $A$ obtained by the normal unfolding a \emph{normal square} or \emph{NS} matrix of $\A$.  We have 

\begin{thm}\label{th:th02}
Let $\A,\B\in \T_{m;n}$ where $m=2k$ is an even number, and $A,B$ are respectively the NS matrices of $\A$ and $\B$.  Then $\A$ is invertible if and only if 
$A$ is invertible. Furthermore, $\B=\A^{-1}$ if and only if $B=A^{-1}$.   
\end{thm}

\begin{proof}
We take $m=4$ for our convenience in notations. The argument in general case (i.e. $m=2k$) follows the same route. Write $A=(a_{ij}), B=(b_{ij})$. 
Then our result follows by   
\beyy
(\A\B)_{i_{1}i_{2}i_{3}i_{4}} &=& \sum\limits_{j_{1},j_{2}} A_{i_{1}i_{2}j_{1}j_{2}} B_{j_{1}j_{2} i_{3} i_{4}} \\
                                           &=& \sum\limits_{j_{1},j_{2}} a_{i_{2}+(i_{1}-1)n, j_{2}+(j_{1}-1)n} b_{j_{2}+(j_{1}-1)n, i_{4}+(i_{3}-1)n}\\
                                           &=& \sum\limits_{s=1}^{n^{2}} a_{is} b_{sj} = (AB)_{ij} 
\eeyy
where $\si:=(i_{1},i_{2},i_{3},i_{4})\in S(4;n)$ and  $i=i_{2}+(i_{1}-1)n, j=i_{4}+(i_{3}-1)n$. 
\end{proof}

\begin{cor}\label{co:co03}
Let $\A\in \T_{m;n}$ where $m=2k$ is an even number.  Then $\A$ is invertible if and only if $\det A \neq 0$ where $A$ is the NS matrix of $\A$.  
\end{cor}

\begin{cor}\label{co:co04}
Let $\A\in \T_{m;n}$ where $m=2k$ is an even number.  Then $\A$ is invertible if and only if there is a tensor $\B\in \T_{m;n}$ such that  $\A\B=\caI$.  
Furthermore, the inverse of $\A$ is unique.   
\end{cor}

\indent  The spectrum theory of tensors is independently introduced by Qi \cite{qieig2005} and Lim \cite{limeig2005} in 2005, and investigated by Qi\cite{qieig2005}, Lim\cite{limeig2005}, and Hu,Huang,Ling and Qi \cite{hhlq2011}. \\
\indent  Let $\A=(A_{i_{1}i_{2}\ldots i_{m}})\in \T_{m;n}$ and $0\neq \bx\in \CC^{n}$.  Then the product $\A\bx^{m-1}\in \CC^{n}$ is a vector. For any number $\la$, if there 
exists a nonzero vector $\bu\in \CC^{n}$ such that 
\beq\label{eq: defeigvalue}
 \A\bu^{m-1} = \la \bu^{[m-1]}
\eeq
where $\bu^{k}$ is an $k$-order $n$-dimensional rank-1 tensor generated by $\bu$, $\bu^{[k]}\in \CC^{n}$ is a vector whose $i$th coordinate is defined as $u_{i}^{k}$ 
where $\bu=(u_{1},u_{2},\ldots,u_{n})^{\top}$.  We call $(\la, \bu)$ an \emph{eigenpair} of $\A$ in which $\la$ is called an \emph{eigenvalue} of $\A$ and $\bu$ is called an \emph{eigenvector} associated with $\la$.  The pair $(\la, \bu)$ is called an H-eigenpair if $\bu$ is a real vector, which is called an H-eigenvector. Note that $\la$ is also a real number in this case and is called a H-eigenvalue of $\A$ if $\A$ is real.  It is shown by Qi et al. \cite{qieig2005} that a tensor $\A\in \T_{m;n}$ (for an even $m$) is positive semidefinite if and only if all its H-(Z-)eigenvalues are nonnegative (see e.g. \cite{qiluo2017}).   \\

\section{Symmetric tensors and anti-symmetric tensors} 
\setcounter{equation}{0}

Let $n>1$ be a positive integer and $m=2k>0$ be an even number.  Denote by $\cP_{m}$ the set of all permutations on set $[m]$.  For any tensor 
$\A=(A_{i_{1}i_{2}\ldots i_{m}})\in \T_{m; n}$ and any permutation $\si\in \cP_{m}$, we define $\si(\A)$ as the tensor $\A^{(\si)}=(A^{(\si)}_{i_{1}i_{2}\ldots i_{m}})$ where 
\beq\label{eq: Asigma}  
A_{i_{1}i_{2}\ldots i_{m}}^{(\si)} = A_{i_{\si(1)}i_{\si(2)}\ldots i_{\si(m)}}, \quad \forall \si:=(i_{1},i_{2},\ldots,i_{m})\in S(m;n). 
\eeq
We call $\A$ \emph{$\si$-symmetric} if it satisfies $\si(\A) = \A$.  $\A$ is called \emph{$\si$-sign symmetric} if $\si(\A) = (-1)^{\tau(\si)} \A$ where $\tau(\si)$ is the inverse number of $\si$.  A tensor $\A\in \T_{m;n}$ is called \emph{anti-symmetric} if $\A$ it is $\si$-sign symmetric for all $\si\in \cP_{m}$.  A symmetric tensor is $\si$-symmetric 
for all $\si\in \cP_{m}$.  Now we denote 
\beq\label{eq: symperm} 
\caS =\frac{1}{\sqrt{m!}} \sum\limits_{\si\in \cP_{m}} \si 
\eeq
then $\caS\colon \T_{m;n}\to \T_{m;n}$ is a linear operator sending each tensor in $\T_{m;n}$ into $\caS_{m;n}$ \cite{cglm2008}. Note that 
\beq\label{eq: permidentitity}
\si\circ \caS = \caS\circ \si =\caS, \quad  \forall  \si\in \cP_{m}.
\eeq
A tensor $\A\in \T_{m;n}$ is symmetric if and only if $\A=\caS(\A)$.  Given a tensor $\A\in \T_{m;n}$.  The \emph{symmetrization} of $\A$ is defined as tensor $\caS(\A)$.  
Note that the polynomial associated with $\A$ is the same as that with $\caS(\A)$, which makes reasonable for us to assume the symmetry of tensors. For our convenience, 
we denote the set of \mnsts by $\st_{m;n}$.\\
\indent  A symmetric tensor $\A\in \st_{m;n}$ is called  \emph{separable}  if 
\beq\label{eq: separableTensor} 
\A = \caS(\bu_{1}\times \bu_{2}\times \ldots \times \bu_{m})  
\eeq
for some vectors $\bu_{1},\bu_{2},\ldots,\bu_{m}\in \RR^{n}$. Some natural questions are: when is a symmetric tensor separable? can a symmetric tensor be decomposed into the sum of some separable tensors? what is the rank of a separable tensor?  \\   
\indent  Now we denote 

Let $\bv_{1},\bv_{2},\cdots, \bv_{m}\in \CC^{n}$ where each $\bv_{j}$ is a nonzero vector, and write  
\[ 
\caL :==\frac{1}{\sqrt{m!}} \sum\limits_{\si\in \cP_{m}} (-1)^{\tau(\si)} \si .
\]
Then $\caL$ is a linear operator on $\T_{m;n}$.  We denote  
\beq\label{eq: asymt} 
\bv_{1}\wedge\cdots\wedge \bv_{m} =\caL(\bv_{1}\times \bv_{2}\times \cdots\times \bv_{m})
\eeq
and 
\beq\label{eq: sepsymt} 
\bv_{1}\vee \bv_{2}\vee \cdots\vee \bv_{m} =\caS(\bv_{1}\times \bv_{2}\times \cdots\times \bv_{m})
\eeq
For $m=2$, the operator $\wedge$ produces an $n\times n$ anti-symmetric matrix of rank 2 when $\bv_{1}, \bv_{2}\in \CC^{n}$ are linearly independent ($n\ge 2$).  
In the following we will show that  tensor $\bv_{1}\wedge\cdots\wedge \bv_{m}$ must be an anti-symmetric tensor in general case.\\
\begin{thm}\label{th3-1}
Let $\bu_{j},\bv_{j}\in \RR^{n}, j\in [m]$ and $A=(a_{ij})\in \RR^{m\times m}$ with $a_{ij}=\seq{\bu_{i},\bv_{j}}$ for all $i,j\in [n]$. Denote 
$\A_{l}=\caL(\bu_{1}\times\cdots\times\bu_{m}), \B_{l}=\caL(\bv_{1}\times\cdots\times\bv_{m})$ and $\A_{s}=\caS(\bu_{1}\times\cdots\times\bu_{m}), \B_{s}=\caS(\bv_{1}\times\cdots\times\bv_{m})$. Then we have 
\beq\label{eq: innerprod-wedge}
\seq{\A_{l}, \B_{l}} =\det(A) 
\eeq
and 
\beq\label{eq: innerprod-sym}
\seq{\A_{s}, \B_{s}} =\perm(A) 
\eeq
where $\perm(A)$ denotes the permanent of matrix $A$. 
\end{thm}

\begin{proof}
For convenience, we denote $\bU:=\bu_{1}\times\cdots\times\bu_{m}, \bV:=\bv_{1}\times\cdots\times\bv_{m}$. We come to prove (\ref{eq: innerprod-wedge}). By definition  
\beyy
\seq{\A_{l}, \B_{l}}&=& \frac{1}{m!} \sum\limits_{\theta} \sum\limits_{\kappa} (-1)^{\tau(\theta)}(-1)^{\tau(\kappa)} \seq{\theta(\bU), \kappa(\bV)}\\
                            &=& \frac{1}{m!} \sum\limits_{\theta,\kappa} (-1)^{\tau(\theta^{-1}\kappa)} \seq{\bU, \theta^{-1}\kappa(\bV)}\\
                            &=& \sum\limits_{\delta\in \cP_{m}} (-1)^{\tau(\delta)} \seq{\bU, \delta(\bV)}\\
                            &=& \sum\limits_{\delta\in \cP_{m}} (-1)^{\tau(\delta)} \seq{\bu_{1},\bv_{\delta(1)}}\seq{\bu_{2},\bv_{\delta(2)}} \cdots \seq{\bu_{m},\bv_{\delta(m)}} \\
                            &=& \sum\limits_{\delta\in \cP_{m}} (-1)^{\tau(\delta)}  a_{1,j_{1}}a_{2,j_{2}}\cdots a_{m,j_{m}}\\
                            &=& \det(A)
\eeyy 
where $\delta:=(j_{1},j_{2},\cdots, j_{m})\in \cP_{m}$ is any permutation of $[m]$ in the last second equation. The proof of formula (\ref{eq: innerprod-sym}) can be deduced similarly.   
\end{proof}

\begin{cor}\label{co3-2}
Let $U=[\bu_{1},\bu_{2},\cdots,\bu_{m}]\in \CC^{n\times m}$ and $\A =\bu_{1}\wedge \cdots\wedge\bu_{m}$. Then we have 
\beq\label{eq: wedgenorm}
\norm{\A} = \si_{1}\si_{2}\ldots \si_{m} 
\eeq
where $\si_{1}\ge \si_{2}\ge \cdots \ge \si_{m}\ge 0$ denote the singular values of matrix $U$, and $\norm{A}$ is the Frobenius norm of a tensor.   
\end{cor}

\begin{proof}
Denote $A=U^{\ast}U$. Then $A=(a_{ij})$ with $a_{ij}=\seq{\bu_{i},\bu_{j}}$ for all $i, j$.  Let $U=QDW^{\ast}$ be the singular value decomposition of $U$ with 
$Q\in \CC^{n\times m}, W\in \CC^{m\times m}$ be column orthogonal and $D=\diag(\si_{1}, \si_{2}, \ldots, \si_{m})$ ($\si_{1}\ge \si_{2}\ge\ldots \ge \si_{m}\ge 0$, here assume that $m\le n$).   Then $\det(U^{\ast}U)=\si_{1}^{2}\si_{2}^{2}\ldots \si_{m}^{2}$.  By Theorem \ref{th3-1}  we have 
\[ \norm{\A} = \sqrt{\seq{\A, \A}} = (\det A)^{1/2} =(\det(U^{\ast}U))^{1/2} = \si_{1}\si_{2}\ldots \si_{m}. \] 
\end{proof}

\begin{cor}\label{co3-3}
Let $\bu_{1},\bu_{2},\cdots,\bu_{m}\in \CC^{n}$ and $\A =\bu_{1}\wedge \cdots\wedge\bu_{m}$.  Then $\A=\caO$ if and only if $\bu_{1},\bu_{2},\cdots,\bu_{m}$ are linear dependent.   
\end{cor}
\begin{proof}
By Corollary \ref{co3-2}, we get $\A=\caO$ if and only if $\norm{\A}=0$ if and only if $\rank(A)<m$ if and only if $\bu_{1},\bu_{2},\cdots,\bu_{m}$ are linear dependent.
\end{proof}

\indent   The following lemma implies that the wedging is a multilinear operation.  

\begin{lem}\label{le: le3-4}
Let $j\in [m]$ and $\bu_{j}, \bw_{j}, \bv_{1}, \bv_{2}, \ldots, \bv_{m}\in \CC^{n}$,  and $\la\in \CC$ be a scalar. Then   
\begin{description}
\item[(1)]  $\bv_{1}\wedge\cdots \wedge (\bu_{j}+\bw_{j}) \cdots\wedge \bv_{m} =\bv_{1}\wedge\cdots \wedge\bu_{j} \cdots\wedge \bv_{m} +\bv_{1}\wedge\cdots \wedge\bw_{j} \cdots\wedge \bv_{m}$.
\item[(2)]  $\bv_{1}\wedge\cdots \wedge (\la\bv_{j})\cdots\wedge \bv_{m} =\la (\bv_{1}\wedge\cdots \wedge\bv_{j} \cdots\wedge \bv_{m})$. 
\end{description}
\end{lem}

\indent  Now we are ready to show that $\bv_{1}\wedge \bv_{2}\wedge \cdots\wedge \bv_{m}$ is an anti-symmetric tensor.  

\begin{thm}\label{th: th3-5} 
Let $\A=\caL(\bv_{1}\times\cdots\times\bv_{m})$ where $\bv_{1},\bv_{2},\cdots, \bv_{m}\in \RR^{n}$ with each $\bv_{j}\neq 0$. Then $\A$ is anti-symmetric. 
\end{thm}

\begin{proof}
For any given $\phi\in \cP_{m}$, we want to show that  $\A^{(\phi)} =(-1)^{\tau(\phi)} \A$. For this purpose, we let $\si=(i_{1},i_{2},\ldots,i_{m})\in S(m;n)$. Then we have 
\beyy
  A^{(\phi)}_{i_{1}i_{2}\ldots i_{m}}  &=& A_{i_{\phi(1)}i_{\phi(2)}\ldots i_{\phi(m)}} \\
                                                       &=& \frac{1}{\sqrt{m!}} \sum\limits_{\si\in \cP_{m}} (-1)^{\tau(\si)} \bv_{\si(\phi(1))}\times \bv_{\si(\phi(2))}\times \cdots\times \bv_{\si(\phi(m))} \\
                                                       &=&(-1)^{\tau(\phi)}\frac{1}{\sqrt{m!}} \sum\limits_{\theta\in \cP_{m}} (-1)^{\tau(\theta)} \bv_{\theta(1)}\times \bv_{\theta(2)}\times \cdots\times \bv_{\theta(m)} \\
                                                       &=&(-1)^{\tau(\phi)} A_{i_{1}i_{2}\ldots i_{m}}.
\eeyy 
Here we denote $\si\phi=\theta$.  Note that $\cP_{m}\phi =\cP_{m}$ for any $\phi\in \cP_{m}$, and that $(-1)^{\tau(\al\beta)} =(-1)^{\tau(\al)}(-1)^{\tau(\beta)}$. 
\end{proof}
 
\indent A tensor $\A\in \T_{m;n}$ is called \emph{separable anti-symmetric} or \emph{SAS} if there exist some vectors $\bv_{1},\bv_{2},\cdots, \bv_{m}\in \RR^{n}$ such that 
$\A=\bv_{1}\wedge\cdots\wedge \bv_{m}$.  Note that $\A=0$ if  $\bv_{j}=0$ for some $j$ or  $\bv_{i}=\bv_{j}$ for some distinct $i,j$. Our next theorem shows that $\A$ is not 
zero only if $\bv_{1},\bv_{2},\cdots, \bv_{m}$ are linearly independent. Here $\caO$ stands for a zero tensor of appropriate size. \\
\indent Let $\bv_{1},\bv_{2},\cdots, \bv_{m}$ be linearly dependent, then there exists a vector $\bv_{j}$ which can be expressed as a linear combination of the others.  
We assume w.l.g. that  $\bv_{m}=\la_{1}\bv_{1}+\ldots +\la_{m-1}\bv_{m-1}$ where $\la_{j}\in \RR$. Corollary \ref{co3-3} can also be deduced by Lemma \ref{le: le3-4} since 
\beyy
 \bv_{1}\wedge\cdots\wedge \bv_{m-1} \wedge \bv_{m} &=& \bv_{1}\wedge\cdots\wedge \bv_{m-1} \wedge (\sum\limits_{j=1}^{m-1} \la_{j}\bv_{j})\\
                                                                                        &=& \sum\limits_{j=1}^{m-1} \la_{j} \bv_{1}\wedge\cdots\wedge \bv_{m-1} \wedge \bv_{j}\\
                                                                                        &=& \caO. 
\eeyy

\indent  An $m$-order $n$-dimensional anti-symmetric tensor $\bv_{1}\wedge\cdots\wedge \bv_{m}$ can be constructed recursively from vectors $\bv_{1},\bv_{2},\cdots, \bv_{m}$.  For this purpose, we define 
\beq\label{eq: bowtie} 
\A\bowtie \bu = \frac{1}{p} \sum\limits_{j=1}^{p} \A\boxtimes_{j} \bu 
\eeq
where $\A\in \T_{p-1; n}, \bu\in \RR^{n}$ and the S-product of $\A$ and $\bu$ in (\ref{eq: bowtie}) is defined with 
\[ S_{1}=\set{1,2,\ldots,j-1, j+1,\ldots,p}, S_{2}=\set{j}. \]
We call the multiplication defined by (\ref{eq: bowtie}) the \emph{bowtie product} of $\A$ and vector $\bu$. 
The bowtie product lift a $(p-1)$-order tensor to an $p$-order tensor. \\  
\indent  Now we are ready to state the bowtie process: 
\begin{thm}\label{th:th3-6}
Let $\A^{(1)}=\bv_{1}, \A^{(2)}=\bv_{1}\wedge \bv_{2}$ and $\A^{(k)}:=\bv_{1}\wedge\bv_{2}\wedge\ldots \wedge \bv_{k}\in \T_{k;n}, k=2,3,\ldots$.  Then the sequence
$\A^{(k)}$ can be constructed through formula  
\beq\label{eq: Akrecursive}
 \A^{(p+1)} =  \A^{(p)} \bowtie \bv_{p+1}, \quad  p=2,3,\ldots,m-1.  
\eeq
\end{thm}
   
\indent The following example gives an expression for a SAS tensor of order 2, i.e., an anti-symmetric matrix, in the sense of separability. 
\begin{exm}\label{exm01}
Let $n\ge 3$ be an integer and $A=\frac{1}{2}(\bu\times \bv -\bv\times \bu)$ where $\bu,\bv\in \RR^{n}$ are linearly independent. Then 
\[  A=\frac{1}{2}(\bu\bv^{\top} -\bv \bu^{\top})\in \RR^{n\times n}  \]
is an anti-symmetric matrix with $\rank(A)=2$. Note that $A$ is not invertible since $n\ge 3$.  
\end{exm}
\indent It is not difficult to show by Example \ref{exm01} that an anti-symmetric nonzero matrix $A$ is separable if and only if $\rank(A)=2$, which implies that 
not all anti-symmetric tensors of order 2 (i.e., anti-symmetric matrices) are separable. \\

\begin{exm}\label{exm02}
Let $m=4, n=2$, and choose a permutation $\si=(2341)$, i.e., $\si(1)=2, \si(2)=3, \si(3)=4, \si(4)=1$. An $\si$-invariant tensor $\A\in \T_{4;2}$ satisfies 
\beyy
A_{1122} = A_{1221}=A_{2211} =A_{2112},  &\qquad & A_{1112} = A_{1121}=A_{1211} =A_{2111}, \\
A_{1222} = A_{2221}=A_{2212} =A_{2121}, &\qquad & A_{1212} = A_{2121}.  
\eeyy 
If we denote by $\Ga(\si)$ and $\Psi_{\si}$ respectively for the set of $\si$-invariant tensors and the set of $\be$-symmetric tensors in $\T_{4;2}$ where $\be=(12)$. 
It is easy to see that both $\Ga(\si)$ and $\Psi_{\be}$ are the subspaces of $\T_{4;2}$, with $\dim(\Ga(\si))=6$ and $\dim(\Psi_{\be})=4$.   
\end{exm}

\vskip 10pt

\section{ The invertibility of separable tensors}
\setcounter{equation}{0}

We denote by $\pi(\A)$ for the spectrum of a tensor $\A$. It is known that a real symmetric matrix $A\in \RR^{n\times n}$ is diagonalizable with all eigenvalues 
being real numbers,  and a real anti-symmetric matrix has zero as its unique real eigenvalue. We conjecture that this phenomenon is also true in the tensor case. In the following we first consider the spectrum of separable symmetric tensors and the spectrum of a separable anti-symmetric tensor. We show that \\
\indent  Let $\A\in \T_{m;n}$.  $\A$ is called a \emph{separable symmetric} tensor if there exist some vectors $\bv_{1},\ldots,\bv_{m}$ such that 
\beq\label{eq:sepsymtensor} 
\A=\caS(\bV)=\frac{1}{\sqrt{m!}} \sum\limits_{\si\in \cP_{m}} \bv_{\si(1)}\times \bv_{\si(2)}\times \cdots\times \bv_{\si(m)}. 
\eeq
where $\bV=\bv_{1}\times\bv_{2}\times \ldots\times\bv_{m}$. A separable tensor is either a separable symmetric tensor or a separable anti-symmetric tensor. In the following 
we write $\bv_{1}\diams\bv_{2}\diams\ldots\diams\bv_{m}$ for $\caL(\bV)$ or $\caS(\bV)$.  Now suppose that $1\le m < n$ and $\A$ is a separable tensor, i.e., 
\beq\label{eq:septensor} 
\A= \bv_{1}\diams \bv_{2}\diams \cdots\diams \bv_{m}. 
\eeq
A natural question is: Is $\A$ invertible or not?  The next result tells us that $\A$ is singular (not invertible) if $m<n$.  
\begin{lem}\label{le4-1}
Let $1\le m<n$, $m=2k$ is an even number, and $\A\in \T_{m;n}$ is invertible. Then $0\notin \pi(\A)$. 
\end{lem}
\begin{proof}
If $0\in \pi(\A)$, there is a nonzero vector $\bx\in \RR^{n}$ which is an eigenvector of $\A$ corresponding to $\la=0$, i.e., 
\beq\label{eq: zeroeig}
\A\bx^{m-1}=0. 
\eeq 
Suppose that $\A$ is invertible, then there exists a tensor $\B\in \T_{m;n}$ such that $\A\B=\B\A=\caI$.  By (\ref{eq: zeroeig}) we have 
\[  \bx^{[m-1]} = (\B\A)\bx^{m-1} = \B(\A\bx^{m-1}) = 0. \]
It follows that $\bx=0$, a contradiction with our assumption. Thus $0\notin \pi(\A)$.     
\end{proof}

\begin{thm}\label{th4-2}
Let $1\le m<n$, $m=2k$ is an even number, and $\A\in \T_{m;n}$ be a separable tensor. Then $\A$ is not invertible.
\end{thm} 

\begin{proof}
First we assume that $\A\in \T_{m;n}$ is separable symmetric. Then $\A=\caS(\bV)=\bv_{1}\vee\bv_{2}\vee\cdots\vee\bv_{m}$ for some $\bv_{j}\in \RR^{n},j\in [m]$, where 
\[ \bV= \bv_{1}\times \bv_{2}\times \cdots\times \bv_{m}. \]
Denote $V=\spann \set{\bv_{1},\bv_{2},\ldots,\bv_{m}}$, i.e., $V$ is the subspace of $\RR^{n}$ spanned by $\bv_{1},\bv_{2},\ldots,\bv_{m}$. If $\bv_{1},\bv_{2}, \ldots,\bv_{m}$ are linearly independent, then $\dim(V)=m<n$.  Given any $\bx\in V^{c}$ where $V^{c}$ is the orthogonal complementary space of $V$, we have 
\beq\label{eq: eq001}  
\A\bx^{m-1} = \caS(\bV)\bx^{m-1} = \caS\left( \prod_{j=2}^{m}(\bv_{j}^{\top}\bx)\bv_{1} \right), 
\eeq
it follows that $\A\bx^{m-1}=\caS(0)=0\in \RR^{n}$ since $\bx\in V^{c}$ is orthogonal to each $\bv_{j}$. Thus $0\in \pi(\A)$. The result is followed by Lemma \ref{le4-1}.     
\end{proof}

\indent Denote by $\bfe_{j}$ the $j$th column vector of the identity matrix $I_{n}$ for $j\in [n]$.  We denote 
\[ \caQ_{n} = \caL(I_{n}):=\caL(\bfe_{1}\times\bfe_{2}\times\cdots \times \bfe_{n}). \]
Then $\caQ=(Q_{i_{1}i_{2}\ldots i_{n}})$ has $n!$ nonzero entries where   
\beq
   Q_{i_{1}i_{2}\ldots i_{n}} =\begin{cases} 1, & \texttt{ if }  (i_{1}, \ldots, i_{n})\in E_{n};\\
                                                                  -1, & \texttt{ if }  (i_{1}, \ldots, i_{n})\in O_{n};\\
                                                                  0,  & \texttt{ otherwise. }     
                                              \end{cases}                    
\eeq
where $E_{n}$ and $O_{n}$ denote respectively the set of even and odd permutations on $[n]$.  For example, $\caQ_{3}$ has six nonzero elements  
\[ Q_{123} = Q_{231} = Q_{312} = -Q_{132} =-Q_{231} =-Q_{321} =1. \]
\indent We call $\caQ_{n}$ the \emph{standard separable anti-symmetric} tensor or \emph{SSAS} tensor.  The following theorem tells us that 
an $n$-order $n$-dimensional real SAS tensor is a scaled SSAS tensor.   
   
\begin{thm}\label{th4-3}
Let $\A\in \T_{n;n}$ be a SAS tensor. Then $\A=\la \caQ$ for some $\la\in \RR$. 
\end{thm}

\begin{proof}
Since $\A$ is a SAS tensor, we may assume that $\A=\caL(A_{1} \times A_{2}\times \cdots \times A_{n})$ where $A=(a_{ij})=[A_{1},\ldots,A_{n}]\in \RR^{n\times n}$. 
Thus we have 
\beyy
 \A  &=& \caL(\sum\limits_{i=1}^{n}a_{i1}\bfe_{i},\sum\limits_{i=1}^{n}a_{i2}\bfe_{i},\cdots, \sum\limits_{i=1}^{n}a_{in}\bfe_{i}) \\
      &=& \sum\limits_{i_{1},i_{2},\ldots,i_{n}} a_{i_{1},1}a_{i_{2},2}\cdots a_{i_{n},n} \caL(\bfe_{i_{1}},\bfe_{i_{2}}, \cdots, \bfe_{i_{n}}) \\
      &=& \caQ \sum\limits_{\si\in \cP_{n}} (-1)^{\tau(\si)} a_{i_{1},1}a_{i_{2},2}\cdots a_{i_{n},n} = \det(A) \caQ
\eeyy
Thus the result holds with $\la=\det(A)$.
\end{proof}

\indent  The following theorem shows that each $3\times 3\times 3$ anti-symmetric tensor is separable.
\begin{thm}\label{th4-4}
Let $\A=(A_{ijk})\in \RR^{3\times 3\times 3}$ be anti-symmetric. Then $\A$ must be separable.
\end{thm}

\begin{proof} 
 If $\A=\caO$, then the statement is true. We suppose that $\A$ is a nonzero tensor. By definition of an anti-symmetric tensor, we know that $\A$ satisfies 
 \[ A_{123} =A_{231} =A_{312} =-A_{132} =-A_{231} =-A_{321}. \]
 and that all other entries shall be zero since the repetitions allowed in their subscripts. Therefore we may assume that   
\beq\label{eq: 333antisym01}  
A_{123} =A_{231} =A_{312} =a, \quad   A_{132} =A_{231} =A_{321} =-a.  
\eeq
where $a\in \RR$. We may assume w.l.g. that $a>0$.  Let $A=(a_{ij})=[\al_{1},\al_{2},\al_{3}]=a^{1/3}I_{3}$ be the scalar matrix of size $3\times 3$.  Write 
$\caE=(E_{ijk})=\al_{1}\times \al_{2}\times \al_{3}$.  Then $\caE$ is a 
rank-one tensor whose unique nonzero entry is $E_{123} = a_{11}a_{22} a_{33}=a$.  Since $\caL(\caE)$ is an anti-symmetric tensor by Theorem \ref{th: th3-5}, its nonzero
entries coincide with that in (\ref{eq: 333antisym01}).  Consequently we have $\A=\caL(\caE)$. This shows that $\A$ is a separable anti-symmetric tensor. 
\end{proof}

\indent  We shall mention that for $n\ge 4$ not all (anti-)symmetric tensors in $\T_{3;n}$ are separable. The following example can be used to illustrate this point.
 
\begin{exm}\label{ex4-5}
Let $\A\in \T_{3;4}$ with its nonzero elements listed as follows:
\bey
  A_{123} =A_{231} =A_{312} =1, &\quad &   A_{132} =A_{213} =A_{321} =-1.\label{t1} \\ 
  A_{124} =A_{241} =A_{412} =2, &\quad &   A_{142} =A_{214} =A_{421} =-2; \label{t2}\\
  A_{134} =A_{341} =A_{413} =3, &\quad &   A_{143} =A_{314} =A_{431} =-3;\label{t3}\\
  A_{234} =A_{342} =A_{423} =1, &\quad &   A_{243} =A_{324} =A_{432} =-1.\label{t4}   
\eey
It is easy to check that this tensor $\A$ is anti-symmetric. We now show that $\A$ is not separable.  Let $\A=\caL(\al\times \be \times \ga)$ for some 
$A=[\al,\be,\ga]\in \RR^{4\times 3}$. Then by Theorem \ref{th4-4} the $3\times 3\times 3$ leading principal sub-tensor $\A_{k}$ (obtained by removing the $k$th layer of each mode) is separable. Furthermore, we have by Theorem \ref{th4-3} $\A_{k}=\caL(\al_{k}\times \be_{k}\times \ga_{k})$ where $\al_{k},\be_{k},\ga_{k}\in \RR^{3}$ are obtained 
respectively by removing the $k$th coordinate of $\al,\be,\ga$.  Thus by the proof of Theorem \ref{th4-4}, we get $A[2:4,:]=I_{3}$ by (\ref{t1}). Similarly, we get 
$A[1:3,:]=I_{3}$ which is conflicted. Thus $\A$ can not be separable.  
\end{exm}

\section{Commutation tensors and the rank of separable tensors}
\setcounter{equation}{0}

\indent  In order to study the rank of a separable tensor, we introduce the definition of the commutation tensors. Recall that a commutation tensor $\K_{p,q}$ is a 4-order 
(0,1)-tensor $\K$ of size $p\times q\times q\times p$, defined as \cite{XHL2020} 
\beq\label{eq:commutationt} 
K_{i_{1}i_{2}i_{3}i_{4}} =1 \iff   i_{1}=i_{4}, i_{2}=i_{3} 
\eeq
It is shown in \cite{XHL2020} that 
\begin{prop}\label{prop5-1}
 For all $\bx\in \RR^{q}, \by\in \RR^{p}$, we have  
\beq\label{eq:prop4commutationt}
\K(\bx\times \by)=\by\times \bx
\eeq 
where the multiplication $\K A$ follows the rule of the contractive product. 
\end{prop}
\indent  The commutation tensor plays a role analog to that of a permutation matrix. Now we extend this definition to a general even order case
\footnote{A general permutation tensor can be defined without the restriction of a constant dimension in the first $m$ modes. Here we simplify it to fit our purpose.}.  
For any positive integer $m>1$ and any given permutation $\si\in \cP_{m}$, we define the permutation tensor $\K^{(\si)}$ as an $2m$-order (0,1)-tensor defined by 
\beq\label{eq: permutationt} 
K^{(\si)}_{i_{1}i_{2}\ldots i_{m},j_{1}j_{2}\ldots j_{m}} =1 \iff   i_{k} = j_{\si(k)}  \forall k\in [m]  
\eeq
\indent  For $m=2$, there are two permutations on set $\set{1,2}$, i.e., 
\begin{description}
\item[(1)]  The identity permutation $\si=(1)(2)$, in which case $\K$ is exactly the identity tensor $\caI$ of order 4. 
\item[(2)]  $\si=(12)$, in which case $\K$ is just the commutation tensor we just mentioned.    
\end{description}

\indent  When $m=3$, there are $3!=6$ permutations on $[3]$. For any permutation $\si\in \set{1,2,3}$, $\K^{(\si)}$ is a 6-order (0,1)-tensor with entry 
\[ K^{(\si)}_{i_{1}i_{2}i_{3}j_{1}j_{2}j_{3}}=1 \quad  \iff \quad  i_{k}=j_{\si(k)}, \forall k=1,2,3. \]
For example, if $\si=(321)$, i.e., $\si(1)=3,\si(3)=2,\si(2)=1$, then $K_{i_{1}i_{2}i_{3} j_{1}j_{2}j_{3}}=1$ iff $i_{1}=j_{3}, i_{2}=j_{1}, i_{3}=j_{2}$. \\
\indent Similar to Proposition \ref{prop5-1}, we have 
\begin{prop}\label{prop5-2}
 Given any permutation $\si\in \cP_{m}$ and a group of vectors $\bu_{1}, \bu_{2}, \ldots, \bu_{m}\in \RR^{n}$, we have  
\beq\label{eq:prop4permt}
\K^{(\si)}(\bu_{1}\times\bu_{2}\times\ldots \times\bu_{m})= \bu_{\si(1)}\times\bu_{\si(2)}\times\ldots \times\bu_{\si(m)}
\eeq 
where the multiplication $\K\A$ follows the rule of the contractive product. 
\end{prop}

\begin{proof}
We denote the tensor of the left hand side and of the right hand side of (\ref{eq:prop4permt}) respectively by $\A$ and $\B$ and write $\bu_{j}=(u_{1j},u_{2j}, \ldots, u_{nj})^{\top}$. Then both $\A,\B\in \T_{m;n}$.  Given $(i_{1},i_{2},\ldots,i_{m})\in S(m;n)$, we have 
\beyy
    A_{i_{1}i_{2}\ldots i_{m}} &=& \sum\limits_{j_{1},j_{2},\ldots,j_{m}} K_{i_{1}i_{2}\ldots i_{m},j_{1}j_{2}\ldots j_{m}} u_{j_{1}1}u_{j_{2}2}\cdots u_{j_{m}m}\\
                                           &=&  u_{i_{1} \si(1)}u_{i_{2} \si(2)}\cdots u_{i_{m} \si(m)}\\
                                           &=& (\bu_{\si(1)}\times\bu_{\si(2)}\times\ldots \times\bu_{\si(m)})_{i_{1}i_{2}\ldots i_{m}}\\
                                           &=& B_{i_{1}i_{2}\ldots i_{m}} 
 \eeyy
Thus $\A=\B$. The proof is completed.   
\end{proof}

\indent Now we are ready to prove 

\begin{thm}\label{th5-3}
Let $\al_{1},\al_{2},\ldots,\al_{m}\in \RR^{n}$ ($1\le m\le n$).  Then the $m!$ vectors in set  
\beq\label{eq: linindept}  
\set{ \al_{\si(1)}\times \al_{\si(2)}\times \ldots \times\al_{\si(m)}:  \si\in \cP_{m} }  
\eeq
are linearly independent if and only if vectors $\al_{1},\al_{2},\ldots,\al_{m}$ are linearly independent.  
\end{thm} 

\begin{proof} 
For sufficiency, we assume that the vectors in set (\ref{eq: linindept}) are linear independent, and we want to prove that vectors $\set{\al_{j}}_{j=1}^{m}$ are 
linear independent. Suppose, to the contrary, that vectors $\set{\al_{j}}_{j=1}^{m}$ are linear dependent. Then by Corollary \ref{co3-2} we have 
$\al_{1}\wedge \ldots \wedge \al_{m}=\caO$, which implies that the vectors in set (\ref{eq: linindept}) are linear dependent, a contradiction to our hypothesis.\\
\indent Now we show the necessity. We assume that vectors $\al_{1},\al_{2},\ldots,\al_{m}$ are linearly independent. To show the linear dependency of vectors in set 
(\ref{eq: linindept}), we let 
\beq\label{eq: linindept2}  
\sum\limits_{\si\in \cP_{m}} \la_{\si} \al_{\si(1)}\otimes \al_{\si(2)}\otimes \ldots \otimes\al_{\si(m)} = \caO   
\eeq
where $\caO$ denotes the zero tensor in $\T_{m;n}$ and $\la_{\si}\in \R$ is a scalar. It suffices to show that $\la_{\si}=0$ for each $\si\in \cP_{m}$.  
By Proposition \ref{prop5-2},  (\ref{eq: linindept2}) is equivalent to  
\beq\label{eq: linindept3}  
\caO = (\sum\limits_{\si\in \cP_{m}} \la_{\si} \K^{(\si)} ) \al_{1}\otimes \al_{2}\otimes \ldots \otimes\al_{m}   
\eeq
which implies that 
\beq\label{eq:linindept4}
\sum\limits_{\si\in \cP_{m}} \la_{\si} \K^{(\si)} =\caO
\eeq
It is easy to see that the tensors in the set $S:=\set{\K^{(\si)}\colon \si\in \cP_{m}}$ are linearly independent. In fact, this assertion can be easily confirmed if we consider
the set of matrices $A^{\si}$ where each $A^{\si}\in \RR^{n^{m}\times n^{m}}$ is obtained by matricization of tensor $\K^{\si}$ in set $S$, i.e., 
$A^{\si}=(A_{ij})$ with $A_{ij}=K_{i_{1}i_{2}\ldots i_{m},j_{1}j_{2}\ldots j_{m}}$ where 
\beq\label{eq: multiindex2matrx}
 i = i_{m}+ \sum\limits_{r=1}^{m-1} (i_{r}-1)n^{k-r},  \quad  j = j_{m} + \sum\limits_{r=1}^{m-1} (j_{r}-1)n^{k-r}.   
\eeq   
Thus $\sum\limits_{\si\in \cP_{m}} \la_{\si} \K^{(\si)}=\caO$  implies that $\la_{\si}=0$ for each $\si\in \cP_{m}$.  Thus the proof is completed.  
\end{proof}

\indent  We note that Theorem \ref{th5-3} is also true if the tensor-products in set (\ref{eq: linindept}) are replaced by the Kronecker products. From Theorem \ref{th5-3},
we immediately get 

\begin{cor}\label{cor5-4}
Let $\al_{1}, \al_{2}, \ldots, \al_{m}\in \RR^{n}$ and $\A=\al_{1}\wedge \al_{2}\wedge\ldots \wedge\al_{m}$.  Then $\rank(\A)\le m!$. Furthermore,  
$\rank(\A)=m!$ if the vectors $\al_{1}, \al_{2}, \ldots, \al_{m}$ are linearly independent.  
\end{cor} 

\indent Corollary \ref{cor5-4} shows that the rank of a tensor can be very large even though its dimension $n$ is small, a fact which is not conformal to the case when $m=2$.
We know that the rank of an $n\times n$ matrix $A$ satisfies $\rank(A)\le n$.  But an 4-order 3-dimensional tensor $\A$ can have rank $4!=24$, which is much bigger than 
3, if we choose $\A$ to be a nonzero separable tensor. Our result can be used to enhance the consequences appeared in \cite{K2009}. \\     
   
\indent  Now we let $\A\in \T_{m;n}$.  We want to know when a symmetric tensor $\A$ can be separable.  In the case $m=2$, we see that $A=\frac{1}{2}(\al\beta^{\top} +\beta\al^{\top})\in \RR^{n\times n}$ is a rank-2 symmetric matrix when $\al,\beta\in \RR^{n}$ are linearly independent.  Furthermore, we can show that 
\begin{lem}\label{le5-5}
Let $\al,\beta\in \RR^{n}$ ($n\ge 2$) be linearly independent, and $A=\frac{1}{2}(\al\beta^{\top} +\beta\al^{\top})$.  Then $A$ is neither positive semidefinite nor negative 
semidefinite.
\end{lem}

\begin{proof}
We may assume w.l.g. that $\norm{\al}=\norm{\beta}=1$ where the norm $\norm{\cdot}$ denote the Euclidean norm.  Denote $a=\seq{\al, \beta}$. Then we have 
$\abs{a} <1$ since $\seq{\al, \beta}^{2}\le \seq{\al, \al}\seq{\beta, \beta}=1$ and the equality holds only if $\al =\pm\beta$, which contradicts with the linear 
independency of $\al, \beta$.  We suppose that $0<a<1$,  and denote $\eta = \frac{1-a}{a}$.  Take $\la =-a(1+2\eta)$ and $\bu=\al- a(1+2\eta)\beta$. Then we can 
easily check that $\bu^{\top}A \bu < 0$. When $-1<a<0$, we can also find a vector $\bv$ such that $\bv^{\top}A \bv < 0$. Thus $A$ is not positive semidefinite. 
Similar argument also applies to show that $A$ is not negative semidefinite.    
\end{proof} 
 
\indent We end the paper by conjecturing that the conclusion in Lemma \ref{le5-5} is also valid for a separable symmetric tensors of order $m\ge 2$.
\begin{conj}
Let $\bv_{1},\bv_{2}, \ldots, \bv_{m}\in \RR^{n}$ ($n\ge m$) be linearly independent vectors, and $\A=\bv_{1}\vee \bv_{2}\vee \ldots \vee\bv_{m}$.  Then $\A$ 
is neither a positive semidefinite tensor nor a negative semidefinite tensor.
\end{conj}

\section*{Acknowledgement}
The author would like to thank Professor Fuzhen Zhang of Nova Southeastern University for his suggestions and remarks for the proof of Theorem \ref{th5-3}.

 %\subsection{}

\end{document}